\documentclass[11pt,letterpaper]{amsart}

\usepackage{amsfonts, amscd, amsmath, mathrsfs, amssymb, amsthm, amsxtra, stmaryrd, bbding, epsfig, graphicx, latexsym, url,xcolor}
\usepackage{xcolor}
\definecolor{cite}{rgb}{0.00,0.60,1.00}
\definecolor{url}{rgb}{1.00,0.10,0.80}
\definecolor{link}{rgb}{0.00,0.00,1.00}
\usepackage[colorlinks,linkcolor=link,urlcolor=url,citecolor=cite,breaklinks]{hyperref}
\usepackage{tikz-cd}

\newtheorem{theorem}[subsection]{Theorem}
\newtheorem{proposition}[subsection]{Proposition}
\newtheorem{corollary}[subsection]{Corollary}
\newtheorem{lemma}[subsection]{Lemma}

\theoremstyle{definition}

\newtheorem{definition}[subsection]{Definition}
\newtheorem{remark}[subsection]{Remark}

\numberwithin{equation}{subsection}

\usepackage{enumerate}

\newcommand{\blambda}{\boldsymbol{\lambda}}
\newcommand{\bk}{\mathbf{k}}
\newcommand{\bm}{\mathbf{m}}
\newcommand{\bx}{\mathbf{x}}

\newcommand{\bN}{\mathbf{N}}
\newcommand{\bR}{\mathbf{R}}
\newcommand{\bT}{\mathbf{T}}
\newcommand{\bZ}{\mathbf{Z}}

\allowdisplaybreaks

\begin{document}

\title{An Effective Deligne's equidistribution theorem}

\author{Lei Fu}
\address{Yau Mathematical Sciences Center, Tsinghua University}
\email{leifu@tsinghua.edu.cn}

\author{Yuk-Kam Lau}
\address{
Weihai Institute for Interdisciplinary Research, Shandong University, China \mbox{\rm and}
Department of Mathematics, The University of Hong Kong, Pokfulam Road, Hong Kong}
\email{yklau@maths.hku.hk}

\author{Ping Xi}
\address{School of Mathematics and Statistics, Xi'an Jiaotong University, Xi'an 710049, P. R. China}
\email{ping.xi@xjtu.edu.cn}



\begin{abstract} We prove an Erd\H{o}s--Tur\'an type inequality for compact Lie groups, from which we deduce
an effective version of Deligne's equidistribution theorem.
\end{abstract}
\maketitle

\section{Introduction}

Let $\mathbf F_q$ be a finite field of characteristic $p$ with $q$ elements, $X$ a smooth geometrically connected 
scheme over $\mathbf F_q$,  $\eta$ the generic point of $X$, 
$\ell$ a prime number distinct from $p$, and
$\mathcal F$ a lisse $\overline{\mathbf Q}_\ell$-sheaf on $X$. 
We have an exact sequence of fundamental groups 
$$1\to \pi_1(X\otimes_{\mathbf F_q}\overline{\mathbf F}_q, \bar \eta)\to \pi_1(X, \bar \eta)
\to\mathrm{Gal}(\overline{\mathbf F}_q/\mathbf F_q)\to 1.$$ 
The lisse $\overline{\mathbf Q}_\ell$-sheaf $\mathcal F$ on $X$ defines a representation
$$\rho_{\mathcal F}: \pi_1(X, \bar \eta)\to \mathrm{GL}(\mathcal F_{\bar \eta}).$$ 
The \emph{geometric monodromy group} (resp. \emph{arithmetic monodromy group}) of $\mathcal F$
is defined to be the Zariski closure of the image of $\pi_1(X\otimes_{\mathbf F_q}\overline{\mathbf F}_q, \bar \eta)$  
(resp. $\pi_1(X, \bar \eta))$ in $\mathrm{GL}(\mathcal F_{\bar \eta})$. 

Fix an isomorphism $\iota: \overline{\mathbf Q}_\ell\stackrel\cong\to \mathbf C$. Throughout this paper, we assume $\mathcal F$ is punctually 
$\iota$-pure of weight $0$, that is, for any closed point $x$ in $X$ and any eigenvalue $\lambda$ of the action of the geometric Frobenius element 
$F_x$ on $\mathcal F_{\bar x}$, we have $|\iota(\lambda)|=1$. Then by theorems of Deligne and Grothendieck 
\cite[3.4.1(iii), 1.3.9]{De80}, the identity component $G^0$ of the geometric monodromy group $G$ is semisimple. Via the isomorphism 
$\iota$, the algebraic group $G$ defines a complex semisimple Lie group $G_{\mathbf C}$.  
The category of finite dimensional 
$\overline{\mathbf Q}_\ell$-representations of $G$ is equivalent to 
that of finite dimensional complex representations of $G_{\mathbf C}$.

We assume $G$ is connected throughout this paper.  Let 
$\mathfrak g$ be the Lie algebra of $G_{\mathbf C}$. Fix a Cartan subalgebra $\mathfrak h$ of $\mathfrak g$, and let 
$R$ be the root system. Choose an ordering $R=R^+\cup R^-$ on $R$. 
For each ${\boldsymbol\alpha}\in R$, let $H_{\boldsymbol\alpha}\in \mathfrak h$ be the element in 
$[\mathfrak g_{\boldsymbol\alpha}, \mathfrak g_{-{\boldsymbol\alpha}}]$ satisfying 
${\boldsymbol\alpha}(H_{\boldsymbol\alpha})=2$, and let $\mathfrak h_{\mathbf R}$ be the real subspace of $\mathfrak h$ spanned by 
$H_{\boldsymbol\alpha}$ $(\boldsymbol\alpha\in R)$. There exists a real 
subalgebra $\mathfrak g_{\mathbf R}$ of $\mathfrak g$ such that 
$$\mathfrak g_{\mathbf R}=i \mathfrak h_{\mathbf R}
\bigoplus \bigoplus_{{\boldsymbol\alpha}\in R^+} 
\Big(\mathfrak g_{\mathbf R}\cap (\mathfrak g_{\boldsymbol\alpha}\oplus\mathfrak g_{-{\boldsymbol\alpha}})\Big),$$ 
that
$\mathfrak g_{\mathbf R}$ is the Lie algebra of a maximal compact subgroup ${G_{\mathbf R}}$ of $G$, and that $i \mathfrak h_{\mathbf R}$ 
is the Lie algebra of a maximal torus $T$ of ${G_{\mathbf R}}$. The category of finite dimensional 
complex representations of $G_{\mathbf C}$ is equivalent to 
that of finite dimensional complex representations of $G_{\mathbf R}$.  

For simplicity, we assume $\rho_{\mathcal F}(\pi_1(X, \bar \eta))\subset G$ as in \cite[3.3]{Ka88}, that is, the arithmetic monodromy 
group coincides with the geometric monodromy group. 
Then for every closed point $x$ in $X$, the geometric Frobenius conjugacy class $F_x$ 
in $\pi_1(U, \bar \eta)^\natural$ defines an element in $G^\natural$ and an element in $G^\natural_{\mathbf C}$, where 
the set of conjugacy classes of a group $H$ is denoted by $H^\natural$. Since 
we assume $\mathcal F$ is punctually $\iota$-pure of weight $0$, the semisimple part of $\rho_{\mathcal F}(F_x)$ 
defines an element in $G_{\mathbf R}^\natural$ 
which we denote by $\rho_{\mathcal F}(F_x)^{\mathrm{ss}}$. Provide $G_{\mathbf R}^\natural$ with the quotient topology induced 
from $G_{\mathbf R}$, and let $\mu_{G_{\mathbf R}^\natural}$ be the measure 
on $G_{\mathbf R}^\natural$ induced by the Haar measure $\mu_{G_{\mathbf R}}$ on $G_{\mathbf R}$. For any integrable function $f$ on $G_{\mathbf R}^\natural$, we have 
$$\int_{G_{\mathbf R}^\natural} f \,\mathrm d\mu_{G_{\mathbf R}^\natural}=
\int_{G_{\mathbf R}} f \,\mathrm d\mu_{G_{\mathbf R}}.$$  
Deligne proves the following theorem (\cite[3.5.3]{De80}, \cite[3.6]{Ka88}).

\begin{theorem}[Deligne] Let $\mathcal F$ be a lisse $\overline{\mathbf Q}_\ell$-sheaf on $X$ punctually 
$\iota$-pure of weight $0$, and let $G$ be 
its geometric monodromy group. Suppose the arithmetic monodromy group of $\mathcal F$ coincides with $G$. 
For any positive integer $m$, let $X(\mathbf F_{q^m})$ be the set of $\mathbf F_{q^m}$-points in $X$.
For any domain $D$ in $G_{\mathbf R}^\natural$ so that the boundary of $D$ has measure $0$, we have
$$\lim_{m\to\infty}\frac{|\{x\in X(\mathbf F_{q^m}):\rho_{\mathcal F}(F_x)^{\mathrm{ss}}\in D \}|}{|X(\mathbf F_{q^m})|}=
\mu_{G_{\mathbf R}^\natural}(D).$$
\end{theorem}

The main result of this paper is an effective version of the above theorem. 

\begin{theorem}\label{thm:effectiveDeligne} Let $\mathcal F$ be a lisse $\overline{\mathbf Q}_\ell$-sheaf on $X$ punctually 
$\iota$-pure of weight $0$, and let $G$ be its geometric monodromy group. Suppose $G$
is connected and the arithmetic monodromy group of $\mathcal F$ coincides with $G$. Then for any 
small box $D$ in $G^\natural_{\mathbf R}$, we have 
\begin{eqnarray*}
\frac{|\{x\in X(\mathbf F_{q^m}):\rho_{\mathcal F}(F_x)^{\mathrm{ss}}\in D\}|}{|X(\mathbf F_{q^m})|}
=\mu_{G_{\mathbf R}^\natural}(D)
+O\Big(q^{-\frac{m}{2(|R^+|+1)}}(\log q^m)^{\frac{n-1}{|R^+|+1}}\Big),
\end{eqnarray*} where $n:=\mathrm{dim}_{\mathbf C}\mathfrak h$ is the rank of the Lie group, 
and the constant implied by $O$ depends only on $X$ and $\mathcal F$. 
\end{theorem}

We will describe small boxes in Definition \ref{special} below. In the case where $X$ is an algebraic curve, let $\overline X$ be a 
smooth compactification of $X$. Then the constant implied by $O$ in the theorem depends only on the group $G$, the genus of $\overline X$, the
number of points in $(\overline X-X)(\overline{\mathbf F}_q)$, the rank of $\mathcal F$ and the Swan conductors of $\mathcal F|_{I_x}$ 
($x\in \overline X-X$), where $I_x$ is the inertia subgroup of $\pi_1(X, \bar\eta)$  at $x$. 

We deduce the above theorem from an Erd\H{o}s-Tur\'an type inequality for compact Lie groups, which is of independent interest
and we now describe. 
Let $$N(T):=\{g\in {G_{\mathbf R}}: gTg^{-1}=T\}$$ be the normalizer of the maximal torus $T$ in ${G_{\mathbf R}}$, and 
let $W:=N(T)/T$ be the \emph{Weyl group}. It acts on $T$ by
$$W\times T\to T, \quad (gT, x)\mapsto gxg^{-1}.$$ Hence $W$ acts on the Lie algebra $\mathfrak h_{\mathbf R}$
and its dual $\mathfrak h_{\mathbf R}^*$. Provide $\mathfrak h_{\mathbf R}$ and $\mathfrak h_{\mathbf R}^*$ with the inner product induced by the Killing form. 
The Weyl group can be identified 
with 
the subgroup of $\mathrm{GL}(\mathfrak h_{\mathbf R}^*)$ generated by 
reflections with respect to ${\boldsymbol\alpha}^\perp$ $({\boldsymbol\alpha}\in R)$.
Let $T/W$ be the quotient space. The map $T\to G_{\mathbf R}^\natural$ sending each element in $T$ to its 
conjugacy class in $G_{\mathbf R}$ defines a homeomorphism $$\kappa:T/W\stackrel\cong\to G_{\mathbf R}^\natural$$
by \cite[IV Proposition 2.6]{BtD85} if we provide $T/W$ and $G_{\mathbf R}^\natural$ with the 
quotient topology induced from $T$ and $G_{\mathbf R}$, respectively. 

The map $\exp: 2\pi i\mathfrak h_{\mathbf R}\to T$ is an epimormphism. Let $2\pi i \Gamma$ be its kernel.
Then we have an isomorphism 
$$\exp(2\pi i\hbox{-}): \mathfrak h_{\mathbf R}/\Gamma\stackrel\cong \to T.$$
We call $\Gamma$ 
the \emph{integral lattice}. Define the \emph{lattice of integral forms} $I^*$ to be
$$I^*:=\{{\boldsymbol\lambda}\in \mathfrak h_{\mathbf R}^*:\, {\boldsymbol\lambda}(\Gamma)\subset \mathbf Z \}.$$ 
For any ${\boldsymbol\lambda}\in I^*$, define a character 
$$e({\boldsymbol\lambda}): T\to S^1:=\{z\in\mathbf C:\, |z|=1\}$$
which maps $\exp(2\pi i H)\in T$ to $e^{2\pi i{\boldsymbol\lambda}(H)}\in S^1$ for every $H\in \mathfrak h_{\mathbf R}$.
We thus get an isomorphism $$I^*\stackrel\cong\to\mathrm{Hom}(T, S^1), \quad \boldsymbol \lambda\mapsto e(\boldsymbol\lambda).$$ 
For any weight ${\boldsymbol\lambda}\in \mathfrak h^*$ of a representation of $\mathfrak g$ coming from 
a representation of $G$, we have $\boldsymbol\lambda\in I^*$. 
In particular, for any root ${\boldsymbol\alpha}\in R$, we have $\alpha\in I^*$ and we have a character $e({\boldsymbol\alpha})$ of $T$. 
Define the \emph{weight lattice} to be $$\Lambda:=
\{{\boldsymbol\lambda}\in \mathfrak h_{\mathbf R}^*:\, {\boldsymbol\lambda}(H_{\boldsymbol\alpha})\in \mathbf Z \hbox{ for all } 
{\boldsymbol\alpha}\in R \}.$$
We have 
\begin{align*}
\sum_{\boldsymbol\alpha \in R}\mathbf ZH_{\boldsymbol\alpha}~\subseteq &~\Gamma~\subseteq~
\{H\in\mathfrak h_{\mathbf R}: \, \boldsymbol\alpha(H)\in\mathbf Z\hbox{ for all }\boldsymbol\alpha\in R\},\\
\Lambda~\supseteq &~I^*~ \supseteq~ \sum_{\boldsymbol\alpha \in R}\mathbf Z{\boldsymbol\alpha}  .
\end{align*}

\begin{definition}\label{special} Fix a basis $\{\mathbf e_1, \ldots, \mathbf e_n\}$ for the lattice $\Gamma$. 
A subset of $T$ of the form $$\Omega=\Big\{\exp\Big(2\pi i\sum_{j=1}^n t_j \mathbf e_j\Big):\, t_j\in I_j\Big\}$$ 
for some intervals $I_j$ of lengths $<1$ is called a \emph{box} in $T$. We say the box $\Omega$ is \emph{small} if
for any nonidentity element $\sigma$ in the Weyl group $W$, $\Omega\cap \sigma(\Omega)$ is empty. 
A subset $D$ of $T/W$ (resp. $G_{\mathbf R}^\natural$) is called a \emph{small box} 
if it is the image of a small box $\Omega$ in $T$ under the canonical map 
$$T\to T/W\quad (\mathrm{resp.} ~T\to T/W\underset\cong{\stackrel\kappa \to}G_{\mathbf R}^\natural).$$   
\end{definition}

\begin{remark} With the above notations, since the lengths of the intervals $I_j$ are $<1$, the map 
$\exp(2\pi i \hbox{-}): \mathfrak h_R\to T$ identifies the set $I_1\times\cdots \times I_n$ with its image 
$\Omega$ in $T$. Suppose furthermore that $\Omega\cap \sigma(\Omega)$ is empty for every nonidentity element 
$\sigma\in W$. Then 
the projection $T\to T/W$ identifies $\Omega$ with its image $D$ in $T/W$, and the inverse image of 
$D$ in $T$ is $\bigsqcup_{\sigma\in W} \sigma(\Omega)$. 
\end{remark}

Let 
$$\mathfrak C:=\{{\boldsymbol\lambda}\in \mathfrak h_{\mathbf R}^*:\; 
{\boldsymbol\lambda}(H_{\boldsymbol\alpha})\in \mathbf R_{\geq 0}\hbox{ for all }{\boldsymbol\alpha}\in R^+\}$$ be the 
\emph{dominant 
Weyl chamber} in $\mathfrak h_R^*$. 
For any ${\boldsymbol\lambda}\in I^*\cap \mathfrak C$, define 
$$A_{\boldsymbol\lambda}:= \sum_{\sigma\in W} \mathrm{sgn}(\sigma) e(\sigma({\boldsymbol\lambda})),$$ 
where for any $\sigma\in W$, $\mathrm{sgn}(\sigma)$ is defined to be determinant of $\sigma$ considered 
as a linear operator on $\mathfrak h_{\mathbf  R}^*$. Its value is $1$ or $-1$ since $\sigma$ is a composite of 
reflections. Let $${\boldsymbol\rho}=\frac{1}{2}\sum_{{\boldsymbol\alpha}\in R^+} {\boldsymbol\alpha}$$
be half of the sum of positive roots. Then $\frac{A_{{\boldsymbol\lambda}+{\boldsymbol\rho}}}
{A_{\boldsymbol\rho}}$ is a trigonometric polynomial function on $T$ invariant under 
the action of $W$ (\cite[VI 1.6]{BtD85}). It defines a continuous function on $T/W\cong G^\natural_{\mathbf R}$ which we 
also denote by $\frac{A_{{\boldsymbol\lambda}+{\boldsymbol\rho}}}
{A_{\boldsymbol\rho}}$. 

The basis $\{\mathbf e_1,\ldots, \mathbf e_n\}$ of $\Gamma$ in Definition \ref{special}
is also a basis of $\mathfrak h_{\mathbf R}$, which we fix henceforth. Let $\{\mathbf e_1^*,\ldots, \mathbf e_n^*\}$ be the dual basis for $\mathfrak h^*_{\mathbf R}$. 
For any $\boldsymbol\lambda\in \mathfrak h^*_{\mathbf R}$, write  $\boldsymbol{\boldsymbol\lambda}=\sum_{j=1}^n {\lambda}_j \mathbf e^*_j$ and let
\begin{align}\label{eq:N-supnorm}
{\mathbf  N}(\boldsymbol\lambda):=(|\lambda_1|+1)\cdots(|\lambda_n|+1).
\end{align}
Finally, fix a norm 
\begin{align}\label{eq:norm}
\Vert\cdot \Vert: \mathfrak h_{\mathbf R}
\to \mathbf R_{\geq 0}
\end{align}
on $\mathfrak h_{\mathbf R}.$
The following is a generation of the Erd\H{o}s--Tur\'an 
inequality (\cite[Chapter 1, Corollary 1.1]{Mo94}) to compact Lie groups. 

\begin{theorem}\label{thm:effectiveWeyl} For any sequence $x_1, \ldots, x_N\in G_{\mathbf R}^\natural$, any 
small box $D$ in $G_{\mathbf R}^\natural$
and any positive integer $M$, we have 
\begin{align*}
& \frac{|\{1\leq i\leq N:x_i\in D\}|}{N}-\mu_{G_{\mathbf R}^\natural}(D)\\
\ll& 
\frac{1}{N}\sum_{\substack{\boldsymbol\lambda\in I^*\cap\mathfrak C-\{\mathbf 0\}\\ \Vert \blambda\Vert\leq M}} 
c(\boldsymbol\lambda)\Big
|\sum_{i=1}^N \frac{A_{{\boldsymbol\lambda}+{\boldsymbol\rho}}}{A_{\boldsymbol\rho}}(x_i)\Big| +\frac{1}{M},
\end{align*}
where $c(\boldsymbol\lambda)=\sum_{\sigma\in W}\frac{1}{{\mathbf N}(\sigma(\blambda))}$. For any real number $r$, we have
\begin{eqnarray}\label{scn}
\sum_{\substack{\boldsymbol\lambda\in I^*\cap\mathfrak C-\{\mathbf 0\}\\ \Vert \boldsymbol \lambda\Vert \le M}} c(\boldsymbol\lambda) 
\Vert\boldsymbol\lambda\Vert^r \ll M^r (\log M)^{n-1}.
\end{eqnarray}
The constants implied by $\ll$ depend only on $G$.
\end{theorem}

\begin{remark} For $G=\mathrm{SL}_2$, Theorem \ref{thm:effectiveWeyl} is due to Niederreiter \cite[Lemma 3]{Ni91}.
Rouse and Thorner \cite[Lemma 3.1]{RT17} give another proof of Niederreiter's theorem 
using the Beurling and Selberg majorizing and minorizing functions. We prove Theorem \ref{thm:effectiveWeyl} using 
the general theory of the majorizing and minorizing functions of Colzani-Gigante-Travaglini (\cite{CGT11}). In \cite{R13}, Rosengarten
obtains a version of the Erd\H{o}s--Tur\'an inequality for simply connected compact Lie groups which is different from ours.
\end{remark}

Let $\mathcal F'$ be another lisse $\overline{\mathbf Q}_\ell$-sheaves on $X$ punctually $\iota$-pure of weight $0$, 
$\rho_{\mathcal F'}:\pi_1(X, \bar \eta)\to \mathrm{GL}(\mathcal F'_{\bar\eta})$ the corresponding representation, 
$G'$ the geometric monodromy group, $G'_{\mathbf R}$ a
maximal compact subgroup of $G'_{\mathbf C}$, 
$\mu_{G'^\natural_{\mathbf R}}$ the measure on $G'^\natural_{\mathbf R}$ induced by  
the Haar measure on $G'_{\mathbf R}$.  We have the following theorem on joint distribution.

\begin{theorem}\label{thm:jointdistribution} Let $\mathcal F$ and $\mathcal F'$ be two lisse $\overline{\mathbf Q}_\ell$-sheaves
punctually $\iota$-pure of weight $0$. Suppose their geometric mondromy groups $G$ and $G'$ are connected and 
coincide with their arithmetic monodromy groups, respectively. Suppose furthermore that for any nontrivial 
irreducible $\overline{\mathbf Q}_\ell$-representations $\Gamma$ of $G$ and $\Gamma'$ of $G'$, 
the representation $(\Gamma \circ\rho_{\mathcal F})\otimes (\Gamma'\circ\rho_{\mathcal F'})$ has no nonzero
$\pi_1(X\otimes_{\mathbf F_q}\overline{\mathbf F}_q, \bar\eta)$-invariant. Then 
for any small boxes $D$ in $G^\natural_{\mathbf R}$ and $D'$ in $G'^\natural_{\mathbf R}$, we have 
\begin{align*}
&\ \ \ \ \frac{|\{x\in X(\mathbf F_{q^m}):\, 
\rho_{\mathcal F}(F_x)^{\mathrm{ss}}\in D,\, \rho_{\mathcal F'}(F_x)^{\mathrm ss}\in D'\}|}{|X(\mathbf F_{q^m})|}\\
&=
\mu_{G^\natural_{\mathbf R}}(D)\mu_{G'^\natural_{\mathbf R}}(D')+
O\Big(q^{-\frac{m}{2(|R^+|+|R'^+|+1)}}(\log q^m)^{\frac{n+n'-2}{|R^+|+|R'^+|+1}}\Big),
\end{align*} where $n$ and $n'$ are the dimensions of the Cartan subalgebras of $\mathfrak g$ and $\mathfrak g'$,
and the constant implied by $O$ depends only on $X$, $\mathcal F$ and $\mathcal F'$. 
\end{theorem}

As an application, we prove 
a joint distribution theorem for the Kloosterman sums and the Airy sums. 
For any $x\in \mathbb G_m(\mathbf F_{q^m})\cong \mathbf F^*_{q^m}$, the Kloosterman sum is defined by 
$$\mathrm{Kl}(\mathbf F_{q^m}, x)=
\sum_{z\in \mathbf F^*_{q^m}}e^{\frac{2\pi i}{p}\mathrm{Tr}_{\mathbf F_{q^m}/\mathbf F_p}(z+\frac{x}{z})}.$$
Deligne \cite[Sommes trig. 7.8]{De77} constructs the Kloosterman sheaf $\mathrm{Kl}$ which is a lisse 
$\overline{\mathbf Q}_\ell$-adic sheaf on 
$\mathbb G_m$ punctually $\iota$-pure of weight $1$ such that 
$$\iota\mathrm{Tr}(F_x, \mathrm{Kl}_{\bar x})
=-\mathrm{Kl}(\mathbf F_{q^m},x).$$ 
for all $x\in \mathbb G_m(\mathbf F_{q^m})$. Consider the twist $\mathrm{Kl}(1/2)$ of $\mathrm{Kl}$ by 
$\mathbf Q_\ell(1/2)$ so that  $\mathrm{Kl}(1/2)$ is punctually $\iota$-pure of weight $0$. Katz (\cite[11.1, 11.3]{Ka88}) 
shows that both the arithmetic and the geometric monodromy groups of $\mathrm{Kl}(1/2)$ are 
$\mathrm{SL}_2$. Define an angle 
$\theta(x)$ by 
\begin{align*}
\mathrm{Kl}(\mathbf F_{q^m},x)&= 2 q^{m/2}\cos\theta(x).
\end{align*}

For any $x\in \mathbb A^1(\mathbf F_{q^m})\cong \mathbf F_{q^m}$, the Airy sum is defined by 
$$\mathrm{Ai}(\mathbf F_{q^m}, x)=
\sum_{z\in \mathbf F_{q^m}}e^{\frac{2\pi i}{p}\mathrm{Tr}_{\mathbf F_{q^m}/\mathbf F_p}(z^3+xz)}.$$
Katz (\cite[Theorem 17]{Ka87}) constructs the Airy sheaf $\mathrm{Ai}$ which is a lisse $\overline{\mathbf Q}_\ell$-adic sheaf on 
$\mathbb A^1$ punctually pure of weight $1$ such that 
$$\iota \mathrm{Tr}(F_x, \mathrm{Ai}_{\bar x})
=-\mathrm{Ai}(\mathbf F_{q^m}, x).$$ 
for all $x\in \mathbb A^1(\mathbf F_{q^m})$. 
Katz (\cite[Theorem 17 (iii), Corollary 20]{Ka87}) proves that 
$\mathrm{det}(\mathrm{Ai})$ is geometrically constant, and is given by the character 
$\mathrm{Gal}(\overline{\mathbf F}_q/\mathbf F_q)\to \overline{\mathbf Q}_\ell^*$ mapping the geometric Frobenius element
in $\mathrm{Gal}(\overline{\mathbf F}_q/\mathbf F_q)$ to $\prod_{\chi^3=1,\,\chi\not=1}(-g(\chi, \psi))$,
where the product is taken over all nontrivial multiplicative characters $\chi: 
\mathbf F^*_q\to \overline{\mathbb Q}_\ell^*$ of order dividing $3$ and 
$$g(\chi, \psi)=\sum_{x\in \mathbf F^*_q} \chi(x)\psi(x)$$ is the Gauss sum. Let $\alpha$ be  
a square root of $\prod_{\chi^3=1,\,\chi\not=1}(-g(\chi, \psi))^{-1}$. Let $\mathrm{Ai}(\alpha)$ be the twist of the Airy sheaf
by the character 
$\mathrm{Gal}(\overline{\mathbf F}_q/\mathbf F_q)\to \overline{\mathbf Q}_\ell^*$ sending the geometric Frobenius element
in $\mathrm{Gal}(\overline{\mathbf F}_q/\mathbf F_q)$ to $\alpha$. Then 
$\mathrm{Ai}(\alpha)$ is punctually $\iota$-pure of weight $0$ and $\mathrm{det}(\mathrm{Ai}(\alpha))=1$. 
Katz (\cite[Theorem 19]{Ka87}) shows that both the arithmetic and the geometric monodromy group of 
$\mathrm{Ai}(\alpha)$ is $\mathrm{SL}_2$ if $p>7$. 
Define an angle 
$\theta'(x)$ by 
\begin{align*}
\mathrm{Ai}(\mathbf F_{q^m}, x)&= 2 \alpha^{-m}\cos\theta'(x).
\end{align*}

\begin{corollary} Suppose $p>7$. 
For any subintervals $I, I'\subset [0,\pi],$ we have
\begin{align*}
\frac{1}{q^m-1}\Big|\{x\in \mathbb G_m(\mathbf F_{q^m}): 
\theta(x)\in I,\; \theta'(x)\in I'\}\Big| =\mu_{\mathrm{ST}}(I) \mu_{\mathrm{ST}}(I')
+O(q^{-\frac{m}{6}}),
\end{align*} where $\mu_{ST}=\frac{2}{\pi}\sin^2\theta \,\mathrm d\theta$ is the Sato-Tate measure on $[0,\pi]$. 
\end{corollary}
 
\begin{proof} We apply Theorem \ref{thm:jointdistribution} to the case where $\mathcal F=\mathrm{Kl}(1/2)$ and 
$\mathcal F'=\mathrm{Ai}(\alpha)$. For $\mathfrak{sl}_2$, the dimension of a Cartan subalgebra is 
$$n=n'=1,$$
and the number 
of positive roots is $$|R^+|=|R'^+|=1.$$ 
A maximal compact subgroup of $\mathrm{SL}_2(\mathbf C)$ is $\mathrm{SU}(2)$, and any element in $\mathrm{SU}(2)$
is conjugate to 
$(\begin{smallmatrix}
e^{i\theta}&\\
&e^{-i\theta}
\end{smallmatrix})$
for a unique $\theta\in [0, \pi]$. So $\mathrm{SU}(2)^\natural$ is identified with $[0,\pi]$. 
For any 
$x\in \mathbb G_m(\mathbf F_{q^m})$, the conjugacy class $\rho_{\mathrm{Kl}(1/2)}(F_x)^{\mathrm{ss}}\in \mathrm{SU}(2)^\natural$ 
corresponds to the angle 
$\theta(x)\in [0, \pi]$, and $\rho_{\mathrm{Ai}(\alpha)}(F_x)^{\mathrm{ss}}\in \mathrm{SU}(2)^\natural$ 
corresponds to the angle 
$\theta'(x)\in [0, \pi]$. Using Weyl's integration formula,
one can verify $\mu_{\mathrm{SU}(2)^\natural}$ is identified with the Sato-Tate measure on $[0,\pi]$. 
Confer Proposition \ref{generalST} below. Irreducible representations of $\mathrm{SL}_2$ are symmetric products of 
the standard representation. We claim that 
for any positive integers $k_1, k_2$, $\mathrm{Sym}^{k_1}(\mathrm{Kl})\otimes \mathrm{Sym}^{k_2}(\mathrm{Ai})$ has 
no nonzero $\pi_1(\mathbb G_{m,\overline{\mathbf F}_q},\bar\eta)$-invariant. 
Our assertion then follows from Theorem \ref{thm:jointdistribution}.

By \cite[4.1.3 and 4.1.4]{Ka88}, $\mathrm{Kl}$ is geometrically self dual. So we have 
\begin{align*}
&\ \ \ \ \Big(\mathrm{Sym}^{k_1}(\mathrm{Kl})_{\bar \eta}\otimes \mathrm{Sym}^{k_2}(\mathrm{Ai})_{\bar \eta}\Big)
^{\pi_1(\mathbb G_{m, \overline{\mathbf F}_q},\bar \eta)}\\
&\cong \Big(\mathrm{Sym}^{k_1}(\mathrm{Kl})_{\bar \eta}^\vee\otimes \mathrm{Sym}^{k_2}(\mathrm{Ai})_{\bar \eta}\Big)
^{\pi_1(\mathbb G_{m, \overline{\mathbf F}_q},\bar \eta)}\\
&\cong \mathrm{Hom}_{\pi_1(\mathbb G_{m, \overline{\mathbf F}_q},\bar \eta)}
\Big(\mathrm{Sym}^{k_1}(\mathrm{Kl})_{\bar \eta}, \mathrm{Sym}^{k_2}(\mathrm{Ai})_{\bar \eta}\Big).
\end{align*}
By \cite[Sommes trig. 7.8 (iii)]{De77}, the sheaf $\mathrm{Sym}^{k_1}(\mathrm{Kl})$ 
is ramified at $0$ for all $k_1\geq 1$. But $\mathrm{Sym}^{k_2}(\mathrm{Ai})$ is unramified at $0$.
So $\mathrm{Sym}^{k_1}(\mathrm{Kl})_{\bar \eta}$ and $\mathrm{Sym}^{k_2}(\mathrm{Ai})_{\bar \eta}$ are non-isomorphic 
irreducible representations of $\pi_1(\mathbb G_{m, \overline{\mathbf F}_q},\bar \eta)$. By Schur's lemma, we have 
$$\mathrm{Hom}_{\pi_1(\mathbb G_{m, \overline{\mathbf F}_q},\bar \eta)}\Big(\mathrm{Sym}^{k_1}(\mathrm{Kl})_{\bar \eta}, 
\mathrm{Sym}^{k_2}(\mathrm{Ai})_{\bar \eta}\Big)=0.$$ 
So $\mathrm{Sym}^{k_1}(\mathrm{Kl})\otimes \mathrm{Sym}^{k_2}(\mathrm{Kl})$ has no nonzero 
$\pi_1(\mathbb G_{m,\overline{\mathbf F}_q},\bar\eta)$-invariant. 
\end{proof}

The paper is organized as follows. In Section 2, we adopt the work \cite{CGT11} of Colzani, Gigante and Travaglini 
to our situation. In Section 3, we prove Theorem  \ref{thm:effectiveWeyl} using the Weyl integration formula, the 
Weyl character formula, and the theorem of Colzani-Gigante-Travaglini. 
In Section 4 we deduce Theorems \ref{thm:effectiveDeligne} from Theorem  
\ref{thm:effectiveWeyl} and Deligne's theorem
(the Weil conjecture). In Section 4, we study joint distribution and prove Theorem \ref{thm:jointdistribution}. 

\subsection*{Acknowledgements} We would like to thank Winnie Li and Peng Shan for inspiring discussions. 
LF is supported by 2021YFA1000700, 2023YFA1009703 and NSFC12171261. YKL is supported by GRF (No. 17303619, 17307720). 
PX is supported by NSFC (No. 12025106).

\section{Trigonometric approximations on the torus}

Inspired by Weyl’s criterion of equidistributions and the classical Erd\H{o}s--Tur\'an inequality on $\bR,$
Colzani, Gigante and Travaglini \cite{CGT11} obtain a general form of the Erd\H{o}s--Tur\'an inequality, which covers the torus case
${\mathbf T}^n:={\mathbf R}^n/{\mathbf Z}^n$. 
For our purpose, we make use of \cite[Corollary 1.2 and Corollary 2.9]{CGT11} 
to derive the following proposition, in which we approximate the characteristic function of a box in $\bT^n$ by certain trigonometric polynomials.

\begin{proposition}
\label{thCGT} 
Let $I_j$ $(j=1,\ldots, n)$ be intervals of lengths $<1$, let $E=\prod_{j=1}^n I_j$, and let
$\Omega$ be the image $E$ in $\mathbf T^n$. Denote by $\chi_\Omega$ the characteristic function of $\Omega.$
For any positive integer $M$, there exist trigonometric polynomials 
$$B^\pm(\mathbf x)=\sum_{\substack{\mathbf k=(k_1,\cdots,k_n)\in\mathbf Z^n\\ |k_j|\leq M}}\widehat B^\pm(\mathbf k) e^{2\pi i\mathbf k\cdot\mathbf x}$$ 
such that 
\begin{eqnarray}\label{lowerupperbound}
B^-(\bx)\le \chi_\Omega(\bx) \le B^+(\bx),
\end{eqnarray}
\begin{eqnarray}\label{estimateFouriercoefficient}
\widehat B^\pm(\mathbf 0)=\int_\Omega \mathrm d\mathbf x +O\Big(\frac{1}{M}\Big), \quad
\widehat B^\pm(\mathbf k)=O\Big(\frac{1}{{\mathbf N}(\mathbf k)}\Big)\,
\end{eqnarray}
for all nonzero $\mathbf k =(k_1,\cdots, k_n)\in \mathbf Z^n$,
where
$$
 \bN(\mathbf k) := \prod_{i=1}^n (|k_i|+1).
$$  The constants implied by $O$ in (\ref{estimateFouriercoefficient}) are independent of $\Omega$. 
\end{proposition}

\begin{proof} 
Let $K$ be a nonnegative function on $\mathbf R^n$ with rapid decay at infinity such that its Fourier transform $\widehat{K}$ satisfies
\begin{align}\label{Ktransform}
\widehat{K}(\mathbf{0})= 1,\quad \widehat{K}(\boldsymbol\xi)\le 1\hbox{ if }\vert\boldsymbol\xi\vert<1,\quad
\widehat{K}(\boldsymbol\xi)=0 \hbox{ if }\vert\boldsymbol\xi\vert\ge 1,
\end{align}
where $\vert\boldsymbol\xi\vert=\sqrt{\xi_1^2+\cdots +\xi_n^2}$ for any $\boldsymbol\xi  =(\xi_1,\cdots, \xi_n)\in \mathbf R^n$.
Put
\begin{align*}    
 K_M(\bx)&:= M^n\sum_{\bk \in {\bZ}^n}  K(M(\bx + \bk)).
\end{align*}
This defines a periodic function which admits the Fourier expansion
\begin{align*}    
 K_M(\bx)&= \sum_{\bk \in {\bZ}^n} \widehat{K}(\bk/M) 
 e^{2\pi i\mathbf k\cdot\mathbf x}.
\end{align*}
According to \eqref{Ktransform}, $ K_M$ is a trigonometric polynomial of degree at most $M$ in each variable.
By \cite[Corollary 1.2]{CGT11} and its proof, one may construct $K$ with the above properties such that
\begin{align}\label{majorminor}
K_M*\chi_\Omega -K_M*H_M \le \chi_\Omega \le K_M*\chi_\Omega +K_M*H_M,
\end{align}
where
\begin{align*}    
H_M(\bx) &=\frac14 \psi(2M {\rm dist}(\bx,\partial E+{\mathbf Z}^n)),\\ 
\psi(t)&=4e^{2\pi}\Big(\int_{|\mathbf x|\le 1}K(\mathbf x)\mathrm d\mathbf x\Big)^{-1}\int_{|\mathbf x|\ge t/2}K(\mathbf x)\,\mathrm d\mathbf x,
\end{align*} 
and $0 \leq \psi(t)\ll_\alpha (1+t)^{-\alpha}$ for any $t\geq0$ and $\alpha>0$. 
It follows that the sums
$$ 
\sum_{\bm \in {\mathbf Z}^n} \psi(M {\rm dist}(\bx,\partial E + \bm)) 
= \sum_{\bm \in {\mathbf Z}^n} \psi(M {\rm dist}(\bx +\bm,\partial E)),
$$ are absolutely convergent. 
We have
$$
0\le    H_M(\bx) = \frac14 \psi(2M {\rm dist}(\bx,\partial E +{\mathbf Z}^n)) \le 
\frac14\sum_{\bm \in {\mathbf Z}^n} \psi( 2M {\rm dist}(\bx,\partial E +\bm)).$$
Let $$H_{\Omega, M}(\bx):= \frac14\sum_{\bm \in {\mathbf Z}^n} \psi( 2M {\rm dist}(\bx,\partial E +\bm)).$$
Since $0\le H_M\le H_{\Omega, M}$ and $K_M\ge 0$, we have $0\le K_M*H_M \le K_M*H_{\Omega, M}$. Take
$$B^\pm:=  K_M*\chi_\Omega \pm K_M*H_{\Omega, M}.$$
The inequalities (\ref{lowerupperbound}) then follow from (\ref{majorminor}).

The trigonometric polynomials $B^\pm$ are explicitly given by 
\begin{align}\label{apm}
B^\pm (\bx) = \sum_{k\in \mathbf Z^n,\,\vert\mathbf k\vert \leq M}  
\widehat{K}(\bk/M) (\widehat{\chi}_\Omega(\bk)\pm \widehat{H}_{\Omega,M}(\bk)) e^{2\pi i \bk\cdot \bx}.
\end{align}
where $\widehat{K}$ satisfies \eqref{Ktransform}, $\widehat{\chi}_\Omega(\bk)$ and $\widehat{H}_{\Omega,M}(\mathbf k)$ 
are the Fourier coefficients
\begin{align*}
\widehat{\chi}_\Omega(\bk)&=\int_{\Omega} e^{-2\pi i\bk\cdot \bx}\,d\bx,\\
\widehat{H}_{\Omega,M} (\bk)
&=\int_{\mathbf T^n}  \frac14 \sum_{\bm \in {\mathbf Z}^n} \psi(2M {\rm dist}(\bx,\partial E+\bm)) e^{-2\pi i\bk\cdot \bx}\,d\bx\\
&=\int_{\mathbf R^n}  \frac14\psi(2M {\rm dist}(\bx,\partial E))e^{-2\pi i \bk\cdot \bx}\,d\bx.
\end{align*}
Note that
\begin{align}\label{evaluate0Fourier}
\widehat{\chi}_\Omega(\bk)&=\int_\Omega e^{-2\pi i\bk\cdot\bx} \mathrm d\mathbf x
=\int_{E} e^{-2\pi i \sum_{j=1}^n k_j x_j}
\mathrm dx_1\cdots \mathrm dx_n\\
&=  \prod_{j=1}^n \int_{I_j} e^{-2\pi i k_j x_j}
\mathrm dx_j \ll\frac{1}{\mathbf N(\mathbf k)}.\nonumber
\end{align} 
It follows that  
$$\widehat{K}(\mathbf 0) \widehat{\chi}_\Omega(\mathbf 0)=\int_\Omega \mathrm d\mathbf x,\quad 
\widehat{K}(\bk/M) \widehat{\chi}_\Omega(\bk)=O\Big(\frac{1}{\bN(\mathbf k)}\Big).
$$
To prove the estimate (\ref{estimateFouriercoefficient}), it suffices to show 
$$
\widehat{H}_{\Omega,M}(\mathbf 0)=O\Big(\frac{1}{M}\Big),\quad 
\widehat{H}_{\Omega,M}(\mathbf k) =O\Big(\frac{1}{\bN(\mathbf k)}\Big)
$$ for any nonzero $\mathbf k\in \mathbf Z^n$ with $\vert\mathbf k\vert \leq M$, and the implied constants are independent of $\Omega$.
By  \cite[Lemma 2.9]{CGT11}, we have
\begin{align*}
 &\bigg| \int_{\mathbf R^n} \psi (2M {\rm dist}(\bx,\partial E)) e^{-2\pi i \mathbf k \cdot \bx}\,d\bx\bigg|\\
\le & c \sum_{j=0}^{n-1} \sum_{A(j)\supset \cdots \supset A(1)} (2M)^{j-n} \prod_{l=1}^j \min \bigg\{\sqrt{n}, \big(2\pi \big\vert P_{A(l)} 
\mathbf k\big\vert\big)^{-1}\bigg\},
\end{align*}
where $c$ is a constant independent of $\Omega$, 
each $A(l)$ is an $l$ dimensional subspace which is an intersection of a number of codimension 1 subspaces 
parallel to the faces of the box $E$,
$P_{A(l)} \mathbf k$ is the orthogonal projection of $\mathbf k$ to $A(l)$, and for $j=0$, the inner sum equals the 
product of $(2M)^{-n}$ and the number of vertices of $E$. Hence the inner sum for $j=0$ is
 \begin{align*}
  (2M)^{-n} 2^n= M^{-n}.
 \end{align*}
 For $1\le j\le n-1$, we have $A(l) ={\rm Span}\{\mathbf e_{i_1},\cdots, \mathbf e_{i_l}\}$ for some $1\leq i_1<\cdots<i_l\leq n$, where 
 $\{\mathbf e_1,\cdots, \mathbf e_n\}$ is the standard basis. 
 Clearly we have
 $$
 \vert P_{A(l)} \mathbf k \vert \ge |k_i|\hbox { for all }i \in \{i_1, \ldots, i_\ell\}.
$$ Hence for each chain $A(j) = {\rm Span}\{e_{i_1},\cdots, e_{i_j}\}\supset \cdots \supset A(1)$, we have
\begin{align*}
(2M)^{j-n} \prod_{l=1}^j\min \bigg\{\sqrt{n}, \big(2\pi \big\vert P_{A(l)}  \mathbf k\big\vert \big)^{-1}\bigg\} 
&\le (2M)^{j-n} \prod_{l=1}^j\big(\sqrt n (|k_{i_l}|+1)^{-1}\big)\\
& \le n^{j/2} \bN(\mathbf k)^{-1},
\end{align*} where for the last inequality we use the fact that $|k_i|+1\leq 2M$ for all $i$. 
For $\mathbf k\neq 0$, we have
\begin{align*}
&\sum_{j=0}^{n-1} \sum_{A(j)\supset \cdots \supset A(1)} (2M)^{j-n} \prod_{l=1}^j \min \bigg\{\sqrt{n}, \big(2\pi \big\vert P_{A(l)} 
\mathbf k\big\vert\big)^{-1}\bigg\} \\
\le& M^{-n}+\sum_{j=1}^{n-1} \frac{n!}{(n-j)!}  n^{j/2}  \bN(\mathbf k)^{-1}
\ll  \bN(\mathbf k)^{-1}.
\end{align*} 
For $\mathbf k=0$, we have
$$
(2M)^{j-n} \prod_{l=1}^j\min \bigg\{\sqrt{n}, (2\pi \big\vert P_{A(l)} 0\big\vert )^{-1}\bigg\} 
= (2M)^{j-n}  n^{j/2} \le n^{j/2}  (2M)^{-1}
$$
since $j\le n-1$, which yields
$$\sum_{j=0}^{n-1} \sum_{A(j)\supset \cdots \supset A(1)} (2M)^{j-n} \prod_{l=1}^j \min \bigg\{\sqrt{n}, \big(2\pi \big\vert P_{A(l)} 
\mathbf k\big\vert\big)^{-1}\bigg\} \ll\frac{1}{M}$$
as desired.
\end{proof}

\section{Proof of Theorem \ref{thm:effectiveWeyl}}

Let $\mu_{{G_{\mathbf R}}}$ and $\mu_T$ be the Haar measures on ${G_{\mathbf R}}$ and $T$, respectively. 
Recall that $\mu_{G_{\mathbf R}^\natural}$ is the 
measure on $G_{\mathbf R}^\natural$ induced by $\mu_{G_{\mathbf R}}$. Let $\mu_{T/W}$ be the measure on $T/W$ 
induced by $\mu_T$.  Elements in the Weyl group $W$ permute 
$\{e({\boldsymbol\alpha})\}_{{\boldsymbol\alpha}\in R}$. So $\prod_{{\boldsymbol\alpha}\in R} (1-e({\boldsymbol\alpha}))$ defines a function 
on $T/W$. 

\begin{proposition}\label{generalST}  Via the homeomorphism
$\kappa: T/W\stackrel\cong\to G_{\mathbf R}^\natural$, the measure $\mu_{G_{\mathbf R}^\natural}$ is identified with the measure
$\big(\frac{1}{|W|} \prod_{{\boldsymbol\alpha}\in R} (1-e({\boldsymbol\alpha})) \big) \mu_{T/W}$ on $T/W$. 
\end{proposition}

\begin{proof} Let $\pi: G_{\mathbf R}\to G_{\mathbf R}^\natural$ be the canonical map. For any
integrable function $f$ on $G_{\mathbf R}^\natural$, by the Weyl integration formula \cite[IV 1.11, VI 1.8]{BtD85},
we have 
\begin{align*}
\int_{G_{\mathbf R}^\natural} f \mathrm{d} \mu_{G_{\mathbf R}^\natural} &=\int_{G_{\mathbf R}} f\circ\pi \, \mathrm{d} 
\mu_{G_{\mathbf R}}=\frac{1}{|W|}\int_T (f\circ \pi |_T) \prod_{{\boldsymbol\alpha}\in R} (1-e({\boldsymbol\alpha}))\, \mathrm d\mu_T\\
&=\frac{1}{|W|}\int_{T/W} (f\circ \kappa) \prod_{{\boldsymbol\alpha}\in R} (1-e({\boldsymbol\alpha}))\, \mathrm d\mu_{T/W}.
\end{align*}
This completes the proof.
\end{proof} 

\begin{proposition} The set $\{\frac{A_{{\boldsymbol\lambda}+{\boldsymbol\rho}}}{A_{\boldsymbol\rho}}:
{\boldsymbol\lambda}\in I^*\cap \mathfrak C\}$ is an orthonormal basis for the Hilbert space 
$L^2(T/W, (\frac{1}{|W|} \prod_{{\boldsymbol\alpha}\in R} (1-e({\boldsymbol\alpha})))\mu_{T/W}).$ 
\end{proposition}

\begin{proof}
By \cite[Theorem VI 1.7 (ii)]{BtD85}, any irreducible representation of $G_{\mathbf R}$ has its highest weight lying in $I^*\cap\mathfrak C$,
and for any ${\boldsymbol\lambda}\in I^*\cap\mathfrak C$, there exists an irreducible representation 
$\Gamma_{\boldsymbol\lambda}$ of $G_{\mathbf R}$
unique up to isomorphism with the highest weight ${\boldsymbol\lambda}$. 
By the Peter-Weyl theorem (\cite[Theorem III 3.1]{BtD85}), the characters of 
$\Gamma_{\boldsymbol\lambda}$ $({\boldsymbol\lambda}\in I^*\cap\mathfrak C)$
form an orthonormal basis 
of $L^2(G_{\mathbf R}^\natural, \mu_{G_{\mathbf R}^\natural})$.
Identify $(G_{\mathbf R}^\natural, \mu_{G_{\mathbf R}^\natural})$ with 
$(T/W, (\frac{1}{|W|} \prod_{{\boldsymbol\alpha}\in R} (1-e({\boldsymbol\alpha})) )\mu_{T/W})$. By the Weyl character formula 
\cite[VI 1.7 (ii)-(iii)]{BtD85}, the function on 
$T/W$ defined by the character of $\Gamma_{\boldsymbol\lambda}$ is exactly 
$\frac{A_{{\boldsymbol\lambda}+{\boldsymbol\rho}}}{A_{\boldsymbol\rho}}.$
\end{proof}

 Let $\Vert\cdot \Vert$ be the norm on $\mathfrak h^*_{\mathbf R}$ as chosen in \eqref{eq:norm}.

\begin{lemma}\label{lemma:technical2}  
For any
$\boldsymbol\lambda=\sum_{j=1}^n\lambda_j\mathbf e^*_j,$
there exists a constant $c>0$ such that for any positive integer $L$ and any $\sigma\in W$, if $\max_j|\lambda_j|\leq L$, then 
$\Vert \sigma(\boldsymbol\lambda)\Vert\leq cL$.
\end{lemma}

\begin{proof}
It suffices to take 
$c=\max_{i,\sigma}\sum_{j=1}^n \Vert \sigma(e^*_i)\Vert$.
\end{proof}

For any $\boldsymbol\lambda=\sum_{j=1}^n\lambda_j\mathbf e^*_j\in \mathfrak h^*_{\mathbf R}$, recall that 
$${\mathbf  N}(\boldsymbol\lambda)=(|\lambda_1|+1)\cdots(|\lambda_n|+1)$$
as defined in \eqref{eq:N-supnorm}. 

\begin{lemma}\label {lemma:technical1} Let  $\boldsymbol\lambda, \boldsymbol\lambda'\in\mathfrak h^*_{\mathbf R}$.
We have $${\mathbf N}(\boldsymbol\lambda + \boldsymbol\lambda') \le{\mathbf N}(\boldsymbol\lambda) {\mathbf N}(\boldsymbol\lambda'),
\quad {\mathbf N}(\boldsymbol\lambda + \boldsymbol\lambda')^{-1} \ll_{\boldsymbol\lambda'}  {\mathbf N}(\lambda)^{-1}.
$$
\end{lemma}

\begin{proof} We have
\begin{align*}
{\mathbf N}(\boldsymbol\lambda + \boldsymbol\lambda')=\prod_{i=1}^n (1+|\lambda_i+\lambda_i'|)
\leq\prod_{i=1}^n\big( 1+|\lambda_i| + |\lambda_i'|\big)
\leq {\mathbf N}(\boldsymbol\lambda){\mathbf N}(\boldsymbol\lambda').
\end{align*} 
Replacing the pair $(\boldsymbol\lambda, \blambda')$ by $(\boldsymbol\lambda+ \boldsymbol\lambda', -\boldsymbol\lambda')$, 
we get the second assertion.
\end{proof}

\begin{lemma}\label{lm:Fourier}
Let $I_j$ $(j=1, \ldots, n)$ be intervals of lengths $<1$, and let $\Omega$ be the image of the set 
$\{\sum_{j=1}^n t_j \mathbf e_j:\, t_j\in I_j\}$ under the composite
$$\mathfrak h_{\mathbf R}\to\mathfrak h_{\mathbf R}/ \sum_{j=1}^n \mathbf Z \mathbf e_j
\underset{\cong}{\stackrel{\exp(2\pi i \hbox{-})}\to} T.$$
We have 
\begin{align*}
\int_\Omega e(\boldsymbol\lambda)\mathrm d \mu_T \ll \frac{1}{{\mathbf  N}(\boldsymbol \lambda)},
\end{align*} where 
the constant implied by $\ll$ depends only on $n$. 
\end{lemma}

\begin{proof} Via the basis $\{e_1,\ldots, e_n\}$ of $\Gamma$, we may identify $\mathbf T^n=\mathbf R^n/\mathbf Z^n$ with 
$\mathfrak h_{\mathbf R}/\Gamma$ and with $T$. Our assertion follows from (\ref{evaluate0Fourier}).
\end{proof}

\begin{lemma}\label{lemma:approx} Let $D$ be a small box in $T/W$, and let $M$ be a positive integer.  
There exists two linear combinations 
$$S^{\pm}=\sum_{\substack{\blambda\in I^*\cap\mathfrak C\\ \Vert\blambda\Vert\leq M}}
\widetilde S^{\pm}({\boldsymbol\lambda})
\frac{A_{{\boldsymbol\lambda}+{\boldsymbol\rho}}}{A_{\boldsymbol\rho}}$$
with $\widetilde S^{\pm}({\boldsymbol\lambda})$ being complex numbers such that the following conditions hold:
\begin{enumerate}[(i)]
\item $S^\pm$ are real valued functions on $T/W$, and 
$$S^-\leq \chi_D\leq S^+$$ everywhere on $T/W$, 
where $\chi_D$ is the characteristic function of $D$.
\item $$\widetilde S^{\pm}(\mathbf 0)=\int_D  \frac{1}{|W|} 
\prod_{{\boldsymbol\alpha}\in R} (1-e({\boldsymbol\alpha}))\, \mathrm{d}\mu_{T/W}+O\Big(\frac{1}{M}\Big).$$ 
\item For any nonzero ${\boldsymbol\lambda}\in I^*\cap\mathfrak C$ with $\Vert\blambda\Vert\leq M,$ we have
$$\widetilde S^{\pm}({\blambda})\ll \sum_{\sigma\in W} \frac{1}{{\mathbf N}(\sigma({\blambda}))}.$$
\end{enumerate}
Moreover, the implied constants in (ii)-(iii) are independent of $D$. 
\end{lemma}

\begin{proof} The lattice of integral forms $I^*$ can be identified with the group of characters of $T$. Let 
$R(G_{\mathbf R})$ be the Grothendieck
ring of representations of $G_{\mathbf R}$. The homomorphism 
$$\mathrm{char}: R(G_{\mathbf R})\to \mathbf Z[I^*]$$ sending
each representation $\rho$ of $G_{\mathbf R}$ to the character of $\rho|_T$ induces an isomorphism 
$$R(G_{\mathbf R})\cong \mathbf Z[I^*]^W.$$
For any $L^2$ function $f$ on $T$, we have the Fourier expansion
$$f=\sum_{\boldsymbol\lambda\in I^*} \widehat f(\boldsymbol\lambda) e(\boldsymbol\lambda),$$
where 
$$\widehat f(\boldsymbol\lambda)=\int_T f e(-\boldsymbol\lambda) \mathrm d \mu_T.$$
Choose intervals  $I_j$ $(j=1, \ldots, n)$ of lengths $<1$ so that $D\subset T/W$ is the image of the set 
$$\Omega=\Big\{\exp\Big(2\pi i \sum_{j=1}^n t_j \mathbf e_j\Big):\, t_j\in I_j\Big\}$$ 
under the projection $T\to T/W$, and  $\Omega\cap \sigma(\Omega)=\emptyset$ for every
nonidentity element $\sigma$ in $W$. Let $L=[M/c]$, where $c$ is the constant in Lemma \ref{lemma:technical2}. 

By Theorem \ref{thCGT}, there exist two trigonometric polynomials 
\begin{eqnarray*}
B^\pm=\sum_{\substack{\boldsymbol\lambda\in I^*\\\max_j|\lambda_j| \leq L} }
\widehat B^{\pm}(\boldsymbol\lambda) e(\boldsymbol\lambda)
\end{eqnarray*}
such that 
\begin{eqnarray}\label{property1}
B^-\leq \chi_\Omega\leq B^+,
\end{eqnarray}
\begin{eqnarray}\label{property2} 
\widehat B^{\pm}(\boldsymbol0)=\widehat \chi_\Omega(\boldsymbol0)  + O\big(L^{-1}\big),
\quad \widehat B^{\pm}(\boldsymbol\lambda)= O\Big(\frac{1}{{\mathbf N}(\boldsymbol\lambda)}\Big),
\end{eqnarray} for all nonzero $\boldsymbol\lambda\in I^*$, 
where $\widehat{\chi}_\Omega(\mathbf 0) =\int_\Omega \mathrm d\mu_T$ and 
the constants implied by $O$ are independent of $\Omega$.  
Set
$$S^\pm= \sum_{\sigma\in W} \sigma(B^\pm).$$ 
Then $S^\pm$ are functions on $T$ invariant under $W$. Hence
they define continuous functions on $T/W$ which we denote also by $S^\pm$. 
The inverse image of $D$ under the projection 
$T\to T/W$ is $\bigsqcup_{\sigma\in W} \sigma(\Omega)$.
By (\ref{property1}), we have 
$$S^-\leq \chi_D\leq S^+$$ as functions on $T/W$.  
We have $S^{\pm}\in \mathbf R[I^*]^W$. So we can expand $S^{\pm}$ as a finite linear combination of the 
orthonormal basis 
$\{\frac{A_{{\boldsymbol\lambda}+\boldsymbol\rho}}{A_{\boldsymbol\rho}}:\,{\boldsymbol\lambda}\in I^*\cap\mathfrak C\}$ 
of $L^2(T/W,(\frac{1}{|W|} \prod_{{\boldsymbol\alpha}\in R} (1-e({\boldsymbol\alpha})))\mu_{T/W})$. We 
can write
\begin{eqnarray}\label{expansionS}
S^{\pm}=\sum_{\substack{\boldsymbol\lambda\in I^*\cap\mathfrak C \\ \Vert\blambda\Vert\leq M}} \widetilde S^{\pm}({\boldsymbol\lambda})
\frac{A_{{\boldsymbol\lambda}+{\boldsymbol\rho}}}{A_{\boldsymbol\rho}},
\end{eqnarray}
where 
$$\widetilde S^{\pm}({\boldsymbol\lambda})=\frac{1}{|W| } \int_T S^{\pm} 
\overline{\Big(\frac{A_{{\boldsymbol\lambda}+
{\boldsymbol\rho}}}{A_{\boldsymbol\rho}}\Big)} \prod_{{\boldsymbol\alpha}\in R} (1-e({\boldsymbol\alpha}))\, \mathrm{d}\mu_T.$$
In the expansion (\ref{expansionS}), the summation goes over those $\boldsymbol\lambda\in I^*\cap\mathfrak C$ with 
$\Vert\blambda\Vert\leq M$. This follows from Lemma \ref{lemma:technical2}, and the fact that in the expansion 
$$S^\pm= \sum_{\sigma\in W} \sigma(B^\pm)=\sum_{\sigma\in W}\sum_{\substack{\boldsymbol\lambda\in I^*\\ \max_j|\lambda_j|\leq L} }
\widehat B^{\pm}(\boldsymbol\lambda) e(\sigma(\boldsymbol\lambda)),$$ the inner summation 
goes over those $\boldsymbol\lambda\in I^*$ with $\max_j|\lambda_j|\leq L$.

By \cite[1.7 (iii)]{BtD85}, we have 
$$A_{\boldsymbol\rho}=e({\boldsymbol\rho})\prod_{{\boldsymbol\alpha}\in R^+}(1-e(-{\boldsymbol\alpha})).$$
So we have 
$$A_{\boldsymbol\rho}\overline{A_{\boldsymbol\rho}}=\prod_{{\boldsymbol\alpha}\in R} (1-e({\boldsymbol\alpha})),\quad
\overline{\Big(\frac{A_{{\boldsymbol\lambda}+{\boldsymbol\rho}}}
{A_{\boldsymbol\rho}}\Big)}\prod_{{\boldsymbol\alpha}\in R} (1-e({\boldsymbol\alpha}))
=\overline{A_{{\boldsymbol\lambda}+{\boldsymbol\rho}}}  A_{\boldsymbol\rho}.$$
Using the fact that 
$\overline{(\frac{A_{{\boldsymbol\lambda}+{\boldsymbol\rho}}}{A_{\boldsymbol\rho}})} 
\prod_{{\boldsymbol\alpha}\in R} (1-e({\boldsymbol\alpha}))\, \mathrm{d}\mu_T$
is $W$-invariant on $T$, we calculate $\widetilde S^{\pm}({\boldsymbol\lambda})$ as follows:
\begin{align}\label{evaluateS}
\widetilde S^{\pm}({\boldsymbol\lambda})
&=\frac{1}{|W| } \int_T S^{\pm} 
\overline{\Big(\frac{A_{{\boldsymbol\lambda}+{\boldsymbol\rho}}}{A_{\boldsymbol\rho}}\Big)} \prod_{{\boldsymbol\alpha}\in R} (1-e({\boldsymbol\alpha}))\, \mathrm{d}\mu_T \\
&= \frac{1}{|W| }\sum_{\sigma\in W}\int_T \sigma(B^\pm)
\overline{\Big(\frac{A_{{\boldsymbol\lambda}+{\boldsymbol\rho}}}{A_{\boldsymbol\rho}}\Big)}\prod_{{\boldsymbol\alpha}\in R} (1-e({\boldsymbol\alpha}))\, \mathrm{d}\mu_T\nonumber \\
&= \int_T B^\pm \overline{\Big(\frac{A_{{\boldsymbol\lambda}+{\boldsymbol\rho}}}{A_{\boldsymbol\rho}}\Big)}\prod_{{\boldsymbol\alpha}\in R} (1-e({\boldsymbol\alpha}))\, \mathrm{d}\mu_T \nonumber\\
&= \int_T B^\pm \overline{A_{{\boldsymbol\lambda}+{\boldsymbol\rho}}}  A_{\boldsymbol\rho}\, \mathrm{d}\mu_T \nonumber\\
&= \sum_{\sigma, \tau\in W} \mathrm{sgn}(\sigma\tau) \int_T
B^\pm  e(\tau({\boldsymbol\rho})-\sigma(\boldsymbol\rho)-\sigma({\boldsymbol\lambda}))\, 
\mathrm{d}\mu_T\nonumber\\
&= \sum_{\sigma, \tau\in W} \mathrm{sgn}(\sigma\tau) \cdot \widehat B^\pm  \big(\sigma({\boldsymbol\lambda}) +\sigma(\boldsymbol\rho) -\tau({\boldsymbol\rho})\big).\nonumber
\end{align} 
For any nonzero ${\boldsymbol\lambda}\in I^*\cap\mathfrak C$, we claim that 
$\sigma({\boldsymbol\lambda}) +\sigma(\boldsymbol\rho) -\tau({\boldsymbol\rho})\not=0$. Otherwise, we have
$\tau^{-1}\sigma({\boldsymbol\lambda}+\boldsymbol\rho)=\boldsymbol\rho$. Since both $\boldsymbol\lambda+\boldsymbol\rho$
and $\boldsymbol\rho$ lie in the dominant Weyl chamber, we must have $\tau^{-1}\sigma=1$. But then $\boldsymbol\lambda+\boldsymbol\rho=\boldsymbol\rho$ which contradicts to $\blambda\not=0$. So we have
\begin{align*}
\widetilde S^{\pm}({\boldsymbol\lambda})
&= \sum_{\sigma, \tau\in W} \mathrm{sgn}(\sigma\tau) \cdot 
\widehat B^\pm  \big(\sigma({\boldsymbol\lambda}) +\sigma(\boldsymbol\rho) -\tau({\boldsymbol\rho})\big)\\
&\ll \sum_{\sigma, \tau\in W} \frac{1}{{\mathbf N}\big(\sigma({\boldsymbol\lambda}) 
+\sigma(\boldsymbol\rho) -\tau({\boldsymbol\rho})\big)}\\
&\ll_{\rho, W} \sum_{\sigma} \frac{1}{{\mathbf N}(\sigma({\boldsymbol\lambda}))},
\end{align*} 
where the first inequality follows from (\ref{property2}), and the second inequality follows from Lemma \ref{lemma:technical1}. 
This proves (iii). 
For ${\boldsymbol\lambda}=0$, the calculation (\ref{evaluateS}) shows that 
$$\widetilde S^{\pm}(\mathbf 0)= \int_T B^\pm \overline{A_{\boldsymbol\rho}}  A_{\boldsymbol\rho}\, \mathrm{d}\mu_T.$$
We have 
\begin{align*}
\Big|\widetilde S^{\pm}(\mathbf 0)-\int_D  \frac{1}{|W|} 
&\prod_{{\boldsymbol\alpha}\in R} (1-e({\boldsymbol\alpha}))\, \mathrm{d}\mu_{T/W}\Big|\\
=& \Big| \int_T B^\pm \overline{A_{\boldsymbol\rho}}  A_{\boldsymbol\rho}\, \mathrm{d}\mu_T-\int_\Omega 
\overline{A_{\boldsymbol\rho}}  A_{\boldsymbol\rho} \, \mathrm{d}\mu_{T}\Big|\\
=&\Big| \int_T (B^\pm-\chi_\Omega) \overline{A_{\boldsymbol\rho}}  A_{\boldsymbol\rho}\, \mathrm{d}\mu_T\Big|\\
\le& \Vert  \overline{A_{\boldsymbol\rho}}  A_{\boldsymbol\rho}\Vert_\infty \int_T\vert B^\pm-\chi_\Omega\vert\, \mathrm{d}\mu_T.
\end{align*}
By (\ref{property2}), the last expression is 
$$ \Vert  \overline{A_{\boldsymbol\rho}}  A_{\boldsymbol\rho}\Vert_\infty \int_T (B^+-\chi_\Omega)\mathrm{d}\mu_T
=\Vert  \overline{A_{\boldsymbol\rho}}  A_{\boldsymbol\rho}
\Vert_\infty (\widehat B^+(\mathbf 0)-\widehat\chi_\Omega(\mathbf 0))=O\Big(\frac{1}{M}\Big)$$
in the $``+"$ case, and 
$$ \Vert  \overline{A_{\boldsymbol\rho}}  A_{\boldsymbol\rho}\Vert_\infty \int_T (\chi_\Omega-B^-)\mathrm{d}\mu_T
=\Vert  \overline{A_{\boldsymbol\rho}}  A_{\boldsymbol\rho}
\Vert_\infty (\widehat\chi_\Omega(\mathbf 0)-\widehat B^-(\mathbf 0))=O\Big(\frac{1}{M}\Big)$$
in the $``-"$ case.
\end{proof} 

\subsection{Proof of Theorem \ref{thm:effectiveWeyl}} By Lemma \ref{lemma:approx}, we have
\begin{align*}
&\ \ \ \ \frac{|\{1\leq i\leq N: x_i\in D\}|}{N}-\mu_{G_{\mathbf R}^\natural}(D)
=\frac{1}{N}\sum_{i=1}^N \chi_D(x_i)- \mu_{G_{\mathbf R}^\natural}(D)\\
&\leq \frac{1}{N}\sum_{i=1}^N S^+(x_i)-\mu_{G_{\mathbf R}^\natural}(D)\\
&= \frac{1}{N}\sum_{\substack{\blambda\in I^*\cap\mathfrak C-\{\mathbf 0\}\\ \Vert\blambda\Vert\leq M}}\widetilde S^+({\boldsymbol\lambda})
\sum_{i=1}^N \frac{A_{{\boldsymbol\lambda}+{\boldsymbol\rho}}}{A_{\boldsymbol\rho}}(x_i)+\widetilde S^+(\mathbf 0)-\mu_{G_{\mathbf R}^\natural}(D)\\
&\ll \frac{1}{N}\sum_{\substack{\blambda\in I^*\cap\mathfrak C-\{\mathbf 0\}\\ \Vert \blambda\Vert\leq M}} c(\blambda)
\Big|\sum_{i=1}^N 
\frac{A_{{\boldsymbol\lambda}+{\boldsymbol\rho}}}{A_{\boldsymbol\rho}}(x_i)\Big|+\frac{1}{M},
\end{align*}
where  
$$
c(\boldsymbol\lambda) =  \sum_{\sigma\in W} \frac{1}{{\mathbf N}(\sigma({\blambda}))}.
$$ 
This gives the desired upper bound. We get the lower bound using $S^-$. 

It remains to prove \eqref{scn}. For any $\lambda=\sum_{j=1}^n\lambda_j\mathbf e_j^*$, let $\Vert\blambda\Vert_\infty=\max_j|\lambda_j|$.
Taking $N:=M \sup\{\Vert \sigma(\blambda)\Vert_\infty: \, \sigma\in W,\;\Vert \blambda\Vert =1\}$, we have
\begin{align*}
 \sum_{\substack{\boldsymbol\lambda\in I^*\cap\mathfrak C-\{\mathbf 0\}\\ \Vert\blambda\Vert\le M}} c(\boldsymbol\lambda)\Vert \boldsymbol\lambda\Vert^r
&=
\sum_{\sigma\in W} 
\sum_{\substack{\boldsymbol\lambda\in I^*\cap\mathfrak C-\{\mathbf 0\}\\ \Vert \blambda\Vert\le M}}
\frac{\Vert \boldsymbol\lambda\Vert^r}{{\mathbf N}(\sigma({\blambda}))}
\\
&\le 
\sum_{\sigma\in W} 
\sum_{\substack{\boldsymbol\lambda\in I^*-\{\mathbf 0\}\\ \max_j|\lambda_j|\le N}}
\frac{\Vert \sigma(\boldsymbol\lambda)\Vert^r}{{\mathbf N}({\blambda})}.
\end{align*}
Note that 
\begin{align*}
   \Vert \sigma(\boldsymbol\lambda)\Vert=\Big\Vert \sum_j\lambda_j\sigma(\mathbf e_j^*)\Big\Vert
   \ll \max_j|\lambda_j|.
\end{align*}
We thus have
\begin{align*}
\sum_{\substack{\boldsymbol\lambda\in I^*-\{\mathbf 0\}\\ \max_j|\lambda_j|\le N}}
\frac{\Vert \sigma(\boldsymbol\lambda)\Vert^r}{{\mathbf N}({\blambda})}
&\ll
      \sum_{\substack{\lambda_1,\ldots, \lambda_n\in \mathbf Z\\
      |\lambda_1|\le \cdots \le| \lambda_n|\leq N}}\frac{(|\lambda_n|+1)^{r-1}}{(|\lambda_1|+1)\cdots (|\lambda_{n-1}|+1)}\\
      &\ll N^r (\log N)^{n-1}\ll M^r (\log M)^{n-1}.
\end{align*}
This completes the proof.  

\section{Proof of Theorem \ref{thm:effectiveDeligne}}

Throughout this section, $X$ is a smooth geometrically connected scheme over $\mathbf F_q$, 
and $\eta$ is the generic point of $X$.  
Any lisse $\overline{\mathbf Q}_\ell$-sheaf $\mathcal G$
on $X$ corresponds to a representation $\rho_{\mathcal G}: \pi_1(X, \bar \eta)\to \mathrm{GL}(\mathcal G_{\bar \eta})$. 
We say $\mathcal G$ 
\emph{has no geometric invariant} (resp. \emph{no geometric coinvariant}) if 
$$\mathcal G_{\bar\eta}^{\pi_1(X\otimes_{\mathbf F_q}\overline{\mathbf F}_q,\bar\eta)}=0\quad 
(\hbox{resp. }\mathcal G_{\bar\eta,\pi_1(X\otimes_{\mathbf F_q}\overline{\mathbf F}_q,\bar\eta)}=0).$$
Choose a finite extension $E$ of $\mathbf Q_\ell$ such that $\mathcal G$ comes from lisse a $E$-sheaf $\mathcal G_E$. Let $R$ be the 
discrete valuation ring of $E$, $\kappa$ the residue field, and $\mathcal G_R$ an $R$-sheaf such that $\mathcal G_E=\mathcal G_R\otimes_R E$. 
Choose a finite Galois \'etale covering $X' \to X$ so that $(\mathcal G_R\otimes_R \kappa)|_{X'}$ is a constant sheaf. 

\begin{proposition}\label{prop:Weil} 
Let $\mathcal G$ be a lisse $\overline{\mathbf Q}_\ell$-sheaf on $X$ punctually $\iota$-pure of weight $0$, and let 
$d=\mathrm{dim}\,X$. Suppose
$\mathcal G$ has no geometric invariant.  We have 
\begin{align*}
\Big|\sum_{x\in X(\mathbf F_{q^m})}\iota \mathrm{Tr}(F_x, \mathcal G_{\bar x})\Big| &\le
 C\, q^{\frac{(2d-1)m}{2}}\mathrm{dim}(\mathcal G_{\bar\eta}), 
\end{align*} where $C$ is a constant depending only on the Galois \'etale covering $X'\to X$ described above. 
\end{proposition}

\begin{proof} Let $\mathcal G^\vee=\mathcal Hom(\mathcal G,\overline{\mathbf Q}_\ell)$ be the dual sheaf of $\mathcal G$. 
By the Poincar\'e duality theorem, we have
\begin{eqnarray*}
&&H^{2d}_c(X\otimes_{\mathbf F_q}\overline{\mathbf F}_q, \mathcal G)\cong
\mathrm{Hom}\Big(H^0 (X\otimes_{\mathbf F_q}\overline{\mathbf F}_q,\mathcal G^\vee),  
\overline{\mathbf Q}_\ell(-d)\Big)\\
&\cong& \mathrm{Hom}\Big(\mathcal G_{\bar\eta}^{\vee, \pi_1(X\otimes_{\mathbf F_q}\overline{\mathbf F}_q,\bar\eta)},\overline{\mathbf Q}_\ell(-d)\Big)
\cong \mathcal G_{\bar\eta,\pi_1(X\otimes_{\mathbf F_q}\overline{\mathbf F}_q,\bar\eta)}(-d).
\end{eqnarray*}
Since $\mathcal G$ is punctually $\iota$-pure of weight $0$, it is geometrically semisimple by \cite[3.4.1(iii)]{De80}. Since
$\mathcal G$ has no geometric invariant, it has no geometric coinvariant and hence
$$H^{2d}_c(X\otimes_{\mathbf F_q}\overline{\mathbf F}_q, \mathcal G)=0.$$ 
Combined with the Grothendieck trace formula \cite[Rapport 3.2]{De77}, we get
\begin{eqnarray*}
\sum_{x\in X(\mathbf F_{q^m})} \mathrm{Tr}(F_x, \mathcal G_{\bar x})=\sum_{i=0}^{2d-1}
\mathrm{Tr}\Big(F^m, H^i_c(X\otimes_{\mathbf F_q}\overline{\mathbf F}_q, \mathcal G)\Big),
\end{eqnarray*} 
where $F$ is the geometric Frobenius correspondence. 
By Deligne's theorem \cite[3.3.1]{De80}, 
all eigenvalues of $F$ on $H^i_c(X\otimes_{\mathbf F_q}\overline{\mathbf F}_q,\mathcal G)$ have absolute values $\leq q^{\frac{i}{2}}$
via $\iota$. By \cite[Lemma 1.8]{FW04}, there exists a constant $c$ depending only on the Galois \'etale covering $X'\to X$ such that 
$$\mathrm{dim}\, H^i_c(X\otimes_{\mathbf F_q}\overline{\mathbf F}_q, \mathcal G)\leq c\,\mathrm{dim}(\mathcal G_{\bar\eta}).$$
So we have 
\begin{eqnarray*}
\Big|\sum_{x\in X(\mathbf F_{q^m})}\iota \mathrm{Tr}(F_x, \mathcal G_{\bar x})\Big|& \le&\sum_{i=0}^{2d-1} q^{\frac{i}{2}}\,\mathrm{dim}\, 
 H^i_c(X\otimes_{\mathbf F_q}\overline{\mathbf F}_q, \mathcal G)\\
&\leq&2cd q^{\frac{(2d-1)m}{2}} \mathrm{dim}(\mathcal G_{\bar\eta}).
\end{eqnarray*}
\end{proof}

\begin{remark} Suppose $X$ is a smooth geometrically connected algebraic curve. 
By \cite[1.4.1]{De80}, we have
\begin{align*}
H^0_c(X\otimes_{\mathbf F_q}\overline{\mathbf F}_q, \mathcal G)&\cong 
\begin{cases}
\mathcal G_{\bar\eta}^{\pi_1(X\otimes_{\mathbf F_q}\overline{\mathbf F}_q,\bar\eta)} &\text{if } X \text{ is proper},\\
0&\text{otherwise}.
\end{cases}
\end{align*}
As $\mathcal G$ has no geometric invariant, we have $H^0_c(X\otimes_{\mathbf F_q}\overline{\mathbf F}_q, \mathcal G)=0$ and 
hence $H^1_c(X\otimes_{\mathbf F_q}\overline{\mathbf F}_q, \mathcal G)$ is the only cohomology group which may not vanish. 
By the Grothendieck--Ogg--Shafarevich formula \cite[X 7.12]{Gr77}, we have
\begin{align*}
\mathrm{dim}\, H^1_c(X\otimes_{\mathbf F_q}\overline{\mathbf F}_q,\mathcal G)
&= (2g-2+N)\mathrm{dim}(\mathcal G_{\bar\eta})+ \sum_{x\in \overline X-X}\mathrm{deg}(x)\mathrm{sw}_x(\mathcal G).
\end{align*} So we have
\begin{align*}
\Big|\sum_{x\in X(\mathbf F_{q^m})}\iota \mathrm{Tr}(F_x, \mathcal G_{\bar x})\Big| &\le
  q^{\frac{m}{2}}\Big((2g-2+N)\mathrm{dim}(\mathcal G_{\bar\eta})+ \sum_{x\in \overline X-X}\mathrm{deg}(x)\mathrm{sw}_x(\mathcal G)\Big). 
\end{align*}
This is more explicit than the estimate in Proposition \ref{prop:Weil}.
\end{remark}

\subsection{Proof of Theorem \ref{thm:effectiveDeligne}}
For any nonzero $\boldsymbol\lambda\in I^*\cap\mathfrak C-\{\mathbf 0\}$, let 
$$\Gamma_{\boldsymbol\lambda}: G\to \mathrm{GL}(r, \overline{\mathbf Q}_\ell)$$ 
be the irreducible representation of $G$ with the highest weight $\boldsymbol\lambda$, where $r=\mathrm{dim}(\Gamma_{\boldsymbol\lambda})$. 
Let $\mathcal G_{\boldsymbol\lambda}$ be the $\overline{\mathbf Q}_\ell$-sheaf corresponding to the representation
$$\pi_1(X, \bar\eta)\stackrel{\rho_{\mathcal F}}\to G\stackrel{\Gamma_{\boldsymbol\lambda}}\to \mathrm{GL}(r, \overline{\mathbf Q}_\ell).$$
Since $\Gamma_{\boldsymbol\lambda}$ is nontrivial and 
irreducible, the sheaf $\mathcal G_{\boldsymbol\lambda}$ has no geometric invariant.
Moreover, $\mathcal G_{\boldsymbol\lambda}$ is punctually $\iota$-pure of weight $0$ since $\rho_{\mathcal F}(F_x)^{\mathrm{ss}}$ 
lies in the maximal compact subgroup $G_{\mathbf R}$ of $G$ for every closed point $x$ in $X$. 
By the Weyl character formula, we have 
$$\iota \mathrm{Tr}(F_x, \mathcal G_{\boldsymbol\lambda,\bar x})= \frac{A_{{\boldsymbol\lambda}+{\boldsymbol\rho}}}
{A_{\boldsymbol\rho}}(\rho_{\mathcal F}(F_x)^{\mathrm{ss}}).$$
We claim that there is a common  finite extension $E$ of $\mathbf Q_\ell$ and a finite Galois \'etale covering $X' \to X$ 
such that for each $\boldsymbol\lambda\in I^*\cap\mathfrak C-\{\mathbf 0\}$, there exists an $R$-sheaf 
$\mathcal G_{\blambda,R}$ such that 
$\mathcal G_{\boldsymbol\lambda}$ comes from the $E$-sheaf $\mathcal G_{\boldsymbol\lambda,R}\otimes_RE$
and $\mathcal G_{\boldsymbol\lambda,R}\otimes_R\kappa|_{X'}$ is a constant sheaf. Indeed, we can first choose such a common 
pair $(E, X'\to X)$ for a finite family of weights
$\blambda_1, \ldots, \blambda_m$ which form a family of generators for the monoid $I^*\cap\mathfrak C$. Then the pair 
$(E, X'\to X)$ works for all $\blambda\in I^*\cap\mathfrak C-\{\mathbf 0\}$ since $\Gamma_{\blambda}$ can be realized 
as a sub-representation of a representation constructed by taking tensor products of representations from the family 
$\Gamma_{\blambda_1},\ldots, \Gamma_{\blambda_m}$.
Applying Proposition \ref{prop:Weil} to the sheaf $\mathcal G_{\boldsymbol\lambda}$, we get
\begin{align}\label{eqn1}
\Big|\sum_{x\in X(\mathbf F_{q^m})} \frac{A_{{\boldsymbol\lambda}+{\boldsymbol\rho}}}
{A_{\boldsymbol\rho}}(\rho_{\mathcal F}(F_x)^{\mathrm{ss}})\Big|
&=\Big|\sum_{x\in X(\mathbf F_{q^m})} \iota \mathrm{Tr}(F_x, \mathcal G_{\boldsymbol\lambda,\bar x})\Big|\\
&\ll \mathrm{dim}(\Gamma_{\boldsymbol\lambda})q^{\frac{(2d-1)m}{2}}, \nonumber
\end{align}
By the Weyl dimension formula \cite[VI 1.7 (iv)]{BtD85}, we have
$$\mathrm{dim}(\Gamma_{\boldsymbol\lambda})=\prod_{{\boldsymbol\alpha}\in R^+}\frac{({\boldsymbol\lambda}+{\boldsymbol\rho})(H_{\boldsymbol\alpha})}{{\boldsymbol\rho}(H_{\boldsymbol\alpha})}.$$
It follows that 
\begin{eqnarray}\label{eqn2}
\mathrm{dim}(\Gamma_{\boldsymbol\lambda})\ll \Vert \blambda\Vert ^{|R^+|}
\end{eqnarray}
for any nonzero $\blambda$.
Combining (\ref{eqn1}) and (\ref{eqn2}) with  Theorem \ref{thm:effectiveWeyl}, for any positive integer $M$, we 
have
\begin{align*}
&\frac{|\{x\in X(\mathbf F_{q^m}):\rho_{\mathcal F}(F_x)^{\mathrm{ss}}\in D\}|}{|X(\mathbf F_{q^m})|}-
\mu_{G_{\mathbf R}^\natural}(D)\\
\ll&  \frac{1}{|X(\mathbf F_{q^m})|}\sum_{\substack{\boldsymbol\lambda\in I^*\cap\mathfrak C-\{\mathbf 0\}\\ \Vert \blambda\Vert\leq M}}
c(\blambda)
\Big|\sum_{x\in X(\mathbf F_{q^m})} \frac{A_{{\boldsymbol\lambda}+{\boldsymbol\rho}}}
{A_{\boldsymbol\rho}}(\rho_{\mathcal F}(F_x)^{\mathrm{ss}})\Big| +\frac{1}{M}\\
\ll & \frac{1}{q^{dm}}\sum_{\substack{\boldsymbol\lambda\in I^*\cap\mathfrak C-\{\mathbf 0\} \\ \Vert \blambda\Vert\leq M}}
c(\blambda)\Vert \blambda\Vert^{|R^+|} q^{\frac{(2d-1)m}{2}}+\frac{1}{M}\\
\ll &q^{-\frac{m}{2}} M^{|R^+|} (\log M)^{n-1}  +\frac{1}{M}.
\end{align*}
Taking $M$ to be the integer closest to $q^{\frac{m}{2(|R^+|+1)}}(\log q^m)^{-\frac{n-1}{|R^+|+1}}$. Then 
$\log M\ll \log q^m$, and 
we have 
\begin{align*}
&q^{-\frac{m}{2}} M^{|R^+|} (\log M)^{n-1}  +\frac{1}{M}\\
\ll& q^{-\frac{m}{2}} M^{|R^+|} (\log q^m)^{n-1}  +\frac{1}{M}\\
\ll& q^{-\frac{m}{2}} q^{\frac{m{|R^+|}}{2(|R^+|+1)}}(\log q^m)^{-\frac{(n-1){|R^+|}}{|R^+|+1}} (\log q^m)^{n-1}  +
q^{-\frac{m}{2(|R^+|+1)}}(\log q^m)^{\frac{n-1}{|R^+|+1}}\\
=&2q^{-\frac{m}{2(|R^+|+1)}}(\log q^m)^{\frac{n-1}{|R^+|+1}}
\end{align*}
So we have 
\begin{eqnarray*}
\frac{|\{x\in X(\mathbf F_{q^m}):\rho_{\mathcal F}(F_x)^{\mathrm{ss}}\in D\}|}{|X(\mathbf F_{q^m})|}
=\mu_{G_{\mathbf R}^\natural}(D)
+O\Big(q^{-\frac{m}{2(|R^+|+1)}}(\log q^m)^{\frac{n-1}{|R^+|+1}}\Big).
\end{eqnarray*}

\section{Joint distribution}

Denote data associated to $G'$ by symbols with a superscript  
$'$. For example, for each 
$\boldsymbol \lambda'\in I'^*\cap \mathfrak C'$, let $\lambda'_j$ be the coefficients in the expression of 
$\boldsymbol \lambda'$
as a linear combination of the basis $\{\mathbf e_1^*, \ldots, \mathbf e^*_{n'}\}$ of $\mathfrak h'^*_{\mathbf R}$
dual to the basis $\{\mathbf e_1, \ldots, \mathbf e_n\}$ of $\mathfrak h'_{\mathbf R}$ in the definition of small boxes for $G'^\natural$.
Let $\Vert\cdot\Vert'$ be the chosen norm on $\mathfrak h'^*_{\mathbf R}$.

\begin{proposition}\label{thm:jointeffectiveWeyl} 
For any sequences $x_1,\ldots, x_N \in G_{\mathbf R}^\natural$
and $x'_1,\ldots, x'_N\in G'^\natural_{\mathbf R}$, any small
boxes $D$ in $G_{\mathbf R}^\natural$ and $D'$ in $G'^\natural_{\mathbf R}$,
and any positive integer $M$, we have 
\begin{align*}
&\frac{|\{1\leq i\leq N:x_i\in D,\, x'_i\in D'\}|}{N}-\mu_{G_{\mathbf R}^\natural}(D)\mu_{G'^\natural_{\mathbf R}}(D')\\
&\ll 
\frac{1}{N}\sum_{\substack{\boldsymbol\lambda\in I^*\cap\mathfrak C,\,
\boldsymbol\lambda'\in I'^*\cap\mathfrak C'\\ (\boldsymbol\lambda,\boldsymbol\lambda')
\not=(\mathbf{0},\mathbf{0})\\ \Vert\blambda\Vert,\Vert\blambda'\Vert'\leq M}}c(\blambda,\blambda')\Big
|\sum_{i=1}^N \frac{A_{{\boldsymbol\lambda}+{\boldsymbol\rho}}}{A_{\boldsymbol\rho}}(x_i)
\frac{A_{{\boldsymbol\lambda'}+{\boldsymbol\rho'}}}{A_{\boldsymbol\rho'}}(x'_i)\Big| +\frac{1}{M},
\end{align*}
where 
$$c(\blambda, \blambda')=\sum_{\sigma\in W, \,\sigma'\in W'}\frac{1}{{\mathbf N}(\lambda){\mathbf N}(\lambda')}.$$
For any real numbers $r$ and $r'$, we have
\begin{eqnarray}\label{joint}
  \quad \sum_{\substack{\boldsymbol\lambda\in I^*\cap\mathfrak C,\,
\boldsymbol\lambda'\in I'^*\cap\mathfrak C'\\ (\boldsymbol\lambda,\boldsymbol\lambda')
\not=(\mathbf{0},\mathbf{0})\\ \Vert\blambda\Vert,\Vert\blambda'\Vert'\leq M}} c(\boldsymbol\lambda,\blambda') 
    (\Vert \boldsymbol\lambda\Vert+1)^r(\Vert \boldsymbol\lambda'\Vert'+1)^{r'} \ll M^{r+r'}(\log M)^{n+n'-2}
\end{eqnarray}
The implied constants depend only on $G$ and $G'$.
\end{proposition}

\begin{proof} The first part is just Theorem \ref{thm:effectiveWeyl} for the group $G\times G'$. We only need to verify 
(\ref{joint}) which is stronger than applying Theorem \ref{thm:effectiveWeyl} directly by saving a logarithm factor. 
This is due to the special feature of the 
action of the Weyl group $W\times W'$. 

Following the arguments in the proof of proof of Theorem \ref{thm:effectiveWeyl}, we have
\begin{align*}
&  \sum_{\substack{\boldsymbol\lambda\in I^*\cap\mathfrak C,\,
\boldsymbol\lambda'\in I'^*\cap\mathfrak C'\\ (\boldsymbol\lambda,\boldsymbol\lambda')
\not=(\mathbf{0},\mathbf{0})\\ \Vert\blambda\Vert,\Vert\blambda'\Vert'\leq M}} c(\boldsymbol\lambda,\blambda') 
   (\Vert \boldsymbol\lambda\Vert+1)^r(\Vert \boldsymbol\lambda'\Vert'+1)^{r'}\\
     &=
\sum_{\substack{\sigma\in W\\ \sigma'\in W'}} 
\sum_{\substack{\boldsymbol\lambda\in I^*\cap\mathfrak C,\,
\boldsymbol\lambda'\in I'^*\cap\mathfrak C'\\ (\boldsymbol\lambda,\boldsymbol\lambda')
\not=(\mathbf{0},\mathbf{0})\\ \Vert\blambda\Vert,\Vert\blambda'\Vert'\leq M}}
\frac{(\Vert \boldsymbol\lambda\Vert+1)^r(\Vert \boldsymbol\lambda'\Vert'+1)^{r'}}{{\mathbf N}(\sigma({\blambda}))
{\mathbf N}(\sigma'({\blambda'}))}
\\
&\ll 
\sum_{\substack{\sigma\in W\\ \sigma'\in W'}} 
\sum_{\substack{\boldsymbol\lambda\in I^*,\,
\boldsymbol\lambda'\in I'\\ (\boldsymbol\lambda,\boldsymbol\lambda')
\not=(\mathbf{0},\mathbf{0})\\ \max_{j,j'}\{|\lambda|_j, |\lambda'|_{j'}\}\leq N}}
\frac{(\Vert \sigma(\boldsymbol\lambda)\Vert+1)^r(\Vert \sigma'(\boldsymbol\lambda')\Vert+1)^{r'}}
{{\mathbf N}({\blambda}){\mathbf N}({\blambda'})},
\end{align*}
where $N= M \sup\{\Vert \sigma(\blambda)\Vert_\infty, \Vert \sigma'(\blambda')\Vert_\infty: \sigma\in W,
\sigma'\in W', \Vert \blambda\Vert= \Vert \blambda'\Vert' =1\}$. (Confer the proof Theorem \ref{thm:effectiveWeyl} for the notation.)
We have
\begin{align*}
   \Vert \sigma(\boldsymbol\lambda)\Vert+1
    \ll \max_{1\le j\le n} |\lambda_j|+1, \quad  \Vert \sigma'(\boldsymbol\lambda')\Vert'+1
    \ll \max_{1\le j\le n} |\lambda'_j|+1.
\end{align*}
We infer that  
\begin{align*}
&\sum_{\substack{\boldsymbol\lambda\in I^*,\,
\boldsymbol\lambda'\in I'\\ (\boldsymbol\lambda,\boldsymbol\lambda')
\not=(\mathbf{0},\mathbf{0})\\ \max_{j,j'}\{|\lambda|_j, |\lambda'|_{j'}\}\leq N}}
\frac{(\Vert \sigma(\boldsymbol\lambda)\Vert+1)^r(\Vert \sigma'(\boldsymbol\lambda')\Vert'+1)^{r'}}
{{\mathbf N}({\blambda}){\mathbf N}({\blambda'})}   \\
   &\ll
      \sum_{\substack{\lambda_1,\ldots, \lambda_n, \lambda'_1,\ldots, \lambda'_{n'}\in\mathbf Z\\
      |\lambda_1|\le \cdots \le| \lambda_n|\leq N\\
       |\lambda'_1|\le \cdots \le| \lambda'_{n'}|\leq N}}
       \frac{(|\lambda_n|+1)^{r-1}}{(|\lambda_1|+1)\cdots (|\lambda_{n-1}|+1)} 
       \frac{(|\lambda'_{n'}|+1)^{r'-1}}{(|\lambda’_1|+1)\cdots (|\lambda'_{n'-1}|+1)}\\
      &\ll N^{r+r'} (\log N)^{n+n'-2}\ll M^{r+r'} (\log M)^{n+n'-2}.
\end{align*}
This completes the proof. 
\end{proof}

\subsection{Proof of Theorem \ref{thm:jointdistribution}} 

For any $\boldsymbol\lambda'\in I'^*\cap\mathfrak C'$, let 
$\Gamma_{\boldsymbol\lambda'}$ 
be the irreducible representation of $G'$ with the highest weight $\boldsymbol\lambda'$, and 
let $\mathcal G_{\boldsymbol\lambda'}$ be the $\overline{\mathbf Q}_\ell$-sheaf on $X$ corresponding to the representation
$\Gamma_{\boldsymbol\lambda'}\circ\rho_{\mathcal F'}$. By our assumption, if $(\boldsymbol\lambda,\boldsymbol \lambda')\not=(0,0)$, 
the sheaf $\mathcal G_{\boldsymbol\lambda}\otimes \mathcal G_{\boldsymbol\lambda'}$ has no geometric invariant and 
is punctually $\iota$-pure of weight $0$. 
By the Weyl character formula, we have 
$$\iota \mathrm{Tr}\Big(F_x, (\mathcal G_{\boldsymbol\lambda}\otimes \mathcal G_{\boldsymbol\lambda'})_{\bar x}\Big)
= \frac{A_{\boldsymbol\lambda+\boldsymbol\rho}}
{A_{\boldsymbol\rho}}(\rho_{\mathcal F}(F_x)^{\mathrm{ss}}) \frac{A_{\boldsymbol\lambda'+\boldsymbol\rho'}}
{A_{\boldsymbol\rho'}}(\rho_{\mathcal F'}(F_x)^{\mathrm{ss}}).$$
Applying Proposition \ref{prop:Weil} to the sheaf $\mathcal G_{\boldsymbol\lambda}\otimes \mathcal G_{\boldsymbol\lambda'}$, we get
\begin{eqnarray*}\label{eqn(i)}
\Big|\sum_{x\in X(\mathbf F_{q^m})} \frac{A_{\boldsymbol\lambda+\boldsymbol\rho}}
{A_{\boldsymbol\rho}}(\rho_{\mathcal F}(F_x)^{\mathrm{ss}}) \frac{A_{\boldsymbol\lambda'+\boldsymbol\rho'}}
{A_{\boldsymbol\rho'}}(\rho_{\mathcal F'}(F_x)^{\mathrm{ss}})\Big|
\ll \mathrm{dim}(\Gamma_{\boldsymbol\lambda}) \mathrm{dim}(\Gamma_{\boldsymbol\lambda'})q^{\frac{(2d-1)m}{2}}\nonumber.
\end{eqnarray*}
Note that 
\begin{eqnarray*}\label{eqn(ii)}
\mathrm{dim}(\Gamma_{\boldsymbol\lambda})\mathrm{dim}(\Gamma_{\boldsymbol\lambda'})
\ll (\Vert \blambda\Vert+1)^{|R^+|} (\Vert \blambda'\Vert'+1)^{|R'^+|}.
\end{eqnarray*}
By Proposition \ref{thm:jointeffectiveWeyl}, for any positive integer $M$, we have  
\begin{align*}
&\frac{|\{x\in X(\mathbf F_{q^m}):\, 
\rho_{\mathcal F}(F_x)^{\mathrm{ss}}\in D,\, \rho_{\mathcal F'}(F_x)^{\mathrm ss}\in D'\}|}{|X(\mathbf F_{q^m})|}-\mu_{G_{\mathbf R}^\natural}(D)
\mu_{G'^\natural_{\mathbf R}}(D')\\
\ll&  \frac{1}{|X(\mathbf F_{q^m})|}\sum_{\substack{\boldsymbol\lambda\in I^*\cap\mathfrak C,\,
\boldsymbol\lambda'\in I'^*\cap\mathfrak C'\\
(\boldsymbol\lambda,\boldsymbol\lambda')\not=(\mathbf{0},\mathbf{0})\\
 \Vert\blambda\Vert, \Vert\blambda'\Vert'\leq M}}
c(\blambda,\blambda')
\Big|\sum_{x\in X(\mathbf F_{q^m})} \frac{A_{\boldsymbol\lambda+\boldsymbol\rho}}
{A_{\boldsymbol\rho}}(\rho_{\mathcal F}(F_x)^{\mathrm{ss}}) \frac{A_{\boldsymbol\lambda'+\boldsymbol\rho'}}
{A_{\boldsymbol\rho'}}(\rho_{\mathcal F'}(F_x)^{\mathrm{ss}})\Big| +\frac{1}{M}\\
\ll& \frac{1}{q^{dm}}\sum_{\substack{\boldsymbol\lambda\in I^*\cap\mathfrak C,\,
\boldsymbol\lambda'\in I'^*\cap\mathfrak C'\\
(\boldsymbol\lambda,\boldsymbol\lambda')\not=(\mathbf{0},\mathbf{0})\\
 \Vert \blambda\Vert,\Vert\blambda'\Vert'\leq M}} c(\blambda, \blambda')
(\Vert \blambda\Vert+1)^{|R^+|} (\Vert \blambda'\Vert'+1)^{|R'^+|}q^{\frac{(2d-1)m}{2}}+\frac{1}{M}\\
\ll& q^{-\frac{m}{2}}M^{|R^+|+|R'^+|}(\log M)^{n+n'-2}+\frac{1}{M}.
\end{align*}
Taking $M$ to be the integer closest to $q^{\frac{m}{2(|R^+|+|R'^+|+1)}}(\log q^m)^{-\frac{n+n'-2}{|R^+|+|R'^+|+1}}$.
Then 
$\log M\ll \log q^m$, and 
we have 
\begin{align*}
q^{-\frac{m}{2}} M^{|R^+|+|R'^+|} (\log M)^{n+n'-2}  +\frac{1}{M}
\ll& q^{-\frac{m}{2}} M^{|R^+|+|R'^+|} (\log q^m)^{n+n'-2}  +\frac{1}{M}\\
\ll& q^{-\frac{m}{2(|R^+|+|R'^+|+1)}}(\log q^m)^{\frac{n+n'-2}{|R^+|+|R'^+|+1}}
\end{align*}
So we have 
\begin{eqnarray*}
&&\frac{|\{x\in X(\mathbf F_{q^m}):\, 
\rho_{\mathcal F}(F_x)^{\mathrm{ss}}\in D,\, \rho_{\mathcal F'}(F_x)^{\mathrm ss}\in D'\}|}{|X(\mathbf F_{q^m})|}-\mu_{G_{\mathbf R}^\natural}(D)
\mu_{G'^\natural_{\mathbf R}}(D')\\
&=&
O\Big(q^{-\frac{m}{2(|R^+|+|R'^+|+1)}}(\log q^m)^{\frac{n+n'-2}{|R^+|+|R'^+|+1}}\Big).
\end{eqnarray*}

\bibliographystyle{plainnat}

\end{document}